\documentclass[a4paper,12pt]{amsart}
\usepackage[letterpaper, margin=1.2in]{geometry}
\usepackage{latexsym}
\usepackage{amssymb}
\usepackage{amsmath}
\usepackage{amsfonts}
\usepackage[table]{xcolor}
\usepackage{pgfplots}
\usepackage{color}
\usepackage{txfonts,pxfonts} 
\usepackage{tikz}
\usetikzlibrary{calc,arrows}
\usetikzlibrary{calc}
\usepackage{verbatim}
\newtheorem{thm}{Theorem}[section]

\newtheorem{lem}[thm]{Lemma}

\newtheorem{cor}[thm]{Corollary}

\newtheorem{prop}[thm]{Proposition}

\theoremstyle{definition}

\def\fph{\mathbb{F}_{\ph}}

\newcommand{\Z}{\mathbb Z}
\newcommand{\z}{\mathbb Z}
\newcommand{\Q}{\mathbb Q}

\newcommand{\F}{\mathbb F}

\newcommand{\q}{\mathbb Q}
\def\F{\mathbb{F}}

\newcommand{\p}{\mathfrak{p}}
\def\ol{\overline}
\def\al{\alpha}
\def\la{\lambda}

\def\om{\omega}
\def\th{\theta}
\def\md#1{\ \mbox{\rm(mod }{#1})}
\def\nph#1{N_{\ph}(#1)}
\def\npp#1{N_{\ph}^+(#1)}
\def\ph{\phi}

\newcounter{cs}
\stepcounter{cs}
\newcommand{\casos}{\begin{itemize}}
\newcommand{\fcasos}{\end{itemize}\setcounter{cs}{1}}

\newfont{\tit}{cmr12 scaled \magstep3}
\title{ A note ON  MONOGENEITY of pure number fields}
\author{Lhoussain El Fadil
}
\address{Faculty of Sciences Dhar El Mahraz, P.O. Box  1874 Atlas-Fes, Sidi Mohamed ben Abdellah University,  Morocco
}\email{lhoussain.elfadil@usmba.ac.ma}
\begin{document}


 
\begin{abstract}
{Gassert's paper "A NOTE ON THE MONOGENEITY OF POWER MAPS"  is cited at least by $17$ papers in the context of monogeneity of pure number fields despite some errors that it contains and remarks on it. In this note, we   point out some of these errors, and make some improvements on it.}
\end{abstract}
    \maketitle
\section{Introduction}
  Let $K$ be a number field generated by  a complex root $\al$ of a monic irreducible polynomial $f(x)\in
  \Z[x]$. { Let $\Z_K$ be its ring of integers.
  It is well know that  the ring $\Z_K$ is a free $\Z$-module of rank $n=[K:\Q]$. Let  $(\Z_K:\Z[\al])$ be the index of $\Z[\al]$ in $\Z_K$.   
  For any rational prime $p$, if $p$ does not divide the index $(\Z_K:\Z[\al])$, then thanks to a well-known { theorem of Dedekind},  the factorization of the ideal  $p \Z_K$ can be directly derived from the factorization of ${\overline f(x)}$ over $\F_p$, where ${\overline f(x)}$ is the reduction of $f(x)$ modulo $p$. Besides in 1894, K. Hensel
developed a powerful approach by showing that the primes of $\z_K$
lying above a prime $p$ are in one-to-one correspondence with
monic irreducible factors of $F(X)$ in $\q_p[X]$. For every prime ideal
corresponding to any irreducible factor in $\q_p[X]$, the
ramification index  and the residue degree together are the same as
those of the local field defined  by the irreducible factor
  { If $(\Z_K:\Z[\th])=1$ for  some $\th\in\Z_K$, then  $(1,\th,\cdots,\th^{n-1})$ is a power integral bases of $\Z_K$. In such a case,  the number field $K$ is said to be monogenic and not monogenic otherwise.
The problem of testing the monogeneity of number fields and constructing power integral bases
have been intensively studied these last four decades, mainly by Ga\'al, Nakahara, 
Pohst, and their collaborators (see for instance \cite{AN, Ga, G19, GO, P}). In \cite{F4}, Funakura, studied the integral bases in pure quartic fields.  In \cite{GR4},  Ga\'al and  Remete, calculated the elements of index $1$, for which the coefficients with absolute value $<10^{1000}$ in the integral basis, of pure quartic field generated by $m^{\frac{1}{4}}$ for {$1< m <10^7$}  and $m\equiv 2,3 \md4$.  In \cite{AN6}, Ahmad, Nakahara, and Husnine proved  that  if $m\equiv 2,3 \md4$ and  $m\not\equiv \mp1\md9$, then the sextic number field generated by $m^{\frac{1}{6}}$ is monogenic.
The same authors showed in \cite{AN},   that if $m\equiv 1 \md4$ and $m\not\equiv \mp1\md9$, then the sextic number field generated by $m^{\frac{1}{6}}$ is not monogenic. In \cite{E6}, based on prime ideal factorization, El Fadil showed that  if $m\equiv 1 \md4$ or $m\not\equiv 1\md9$, then the sextic number field generated by $m^{\frac{1}{6}}$ is not monogenic.
In \cite{GR17}, by applying the  explicit form of the index, Ga\'al  and  Remete obtained new results on  monogeneity of the number  fields generated  by $m^{\frac{1}{n}}$ for $3\le n\le 9$. Recall also that in \cite{G}, Gassert studied the integral closedness of $\Z[\al]$, where $\al$ is a complex root of a monic irreducible polynomial $f(x)=x^n-m\in \Z[x]$. Finally, we cite that  in \cite{E24, E12, E36, E23k}, based on Newton polygon techniques, El Fadil studied the monogeneity of pure number fields of degree $24,\, 12,$, $36$, and $2\cdot 3^k$ respectively. 
In  this paper, we  state some comments regarding Gassert's paper \cite{G}, we point out some errors  in \cite[Proposition 3.4]{G}, and so in the proof of \cite[Proposition 3.5 and Theorem 1.2]{G}, we also  improve \cite[Theorem 1.1 and Theorem 1.2]{G}, and we conclude with Theorem \ref{mono}, which gives sufficient conditions on $f(x)=x^n-m$ in order to have $K$  not monogenic, where $K$ is the number field generated by a complex root of $f(x)$.
\section{{Errors in  \cite{G} and comments}}\label{comments}
\begin{enumerate}
\item
In his paper "A NOTE ON THE MONOGENEITY OF POWER MAPS" \cite{G}, Gassert  introduced the notion of monogenic polynomials as follows: a monic polynomial  $f(x)\in \Z[x]$ is said to be monogenic if it is irreducible and $K=\Q(\al)$ is monogenic, that is  the ring of integers $\Z_K=\Z[\th]$ for some element  $\th\in \Z_K$.
In \cite[Theorem 1.1]{G}, he gave sufficient conditions on  $f(x)=x^n-m\in \Z[x]$,  in order to have $f(x)$ is monogenic. He then gave a remark in which he claimed that the hypothesis of irreduciblity of $f(x)$ is not required in \cite[Theorem 1.1]{G} because under the hypotheses: $m$ is a square free integer and $m^p\not\equiv m \md{p^2}$ for every prime integer $p$ dividing $n$, $f(x)$ is $p$-Eisenstein for some prime integer $p$. 
\begin{enumerate}
\item
{ This remark is not true}, it suffices to consider the following example $f(x)=x^{2^k}+1$ for some non-negative integer $k$. It is clear that $f(x)$  is not $p$-Eisenstein for any prime integer $p$. In Lemma \ref{2^k}, we show that $f(x)=x^{2^k}+1$ is irreducible over $\Q$ for every non-negative integer $k$. 
\item
  Even if the content of \cite[Theorem 1.1]{G} was true, it would  give only a partial answer to the problem of mongeneity of $K$. More precisely, it gives sufficient conditions on $f(x)$ to have $\Z_K=\Z[\al]$. But it does not give any information on the existence of  an other  integral element $\th$ which satisfies $\Z_K=\Z[\th]$. As an easy example, for $f(x)=x^{10}- 10^3$, by applying \cite[Theorem 1.1]{G}, $\Z_K\neq \Z[\al]$. For the question, using the above,  could we claim that $K$ is not monogenic?  Note that if we replace $\al$ by $\th=\frac{\al^3}{10^2}$, then $\th\in K$. As $\th^{10}=10$, then $\th$ is a root of $g(x)=x^{10}-10$, which is $2$-Eisenstein. Thus $\th$ is a primitive element of $K$. By applying again \cite[Theorem 1.1]{G}, we get $\Z_K\neq \Z[\th]$.\\
\textcolor[rgb]{0.00,0.00,1.00}{ In our point of view, \cite[Theorem 1.1]{G} does not deal completely with the problem of monogeneity of $K$, but rather, it deals with the problem  of integral closedness of
$\Z [\al]$. The problem of monogeneity is more hard than that concerning  integral closedness, which contributes partially in the study of monogeneity of $K$.} In section \ref{fix}, we show that  \cite[Theorem 1.1]{G} characterizes the integral closedness of
$\Z [\al]$. 
\end{enumerate}
\item
Regarding \cite[Theorem 1.2]{G}, we split its content  in two parts:
\begin{enumerate}
\item
Let $p$ be a prime integer which divides $m$. If $gcd(n, p, \nu_p(m)) = 1$, then $\nu_p(ind(f))=\frac{(n-1)(\nu_p(m)-1)}{2}+d-1$, where $d=$gcd$(n,\nu_p(n))$.\\
 \textcolor[rgb]{0.00,0.00,1.00}{This part is true. But one can wonder what  is the  meaning of   "gcd$(n,p, \nu_p(m)) = 1$?"}\\
  Under this condition, we have $\ol{f(x)}=\ol{\ph}^n$ in $\F_p[x]$ with $\ph=x$ and $\nph{f}=S$ has a single side of degree $d=$gcd$(n,\nu_p(m))$ and $p$ does not divide $d$. In this case, the associated residual polynomial is $f_S(y)=y^d-m_p$ is a separable polynomial over $\F_p=\fph$, where $m_p=\frac{m}{m^{\nu_p(t)}}\md{p}$. Thus according to the terminolgy of Ore \cite{O, EMN}, $f(x)$ is $p$-regular. It follows by Theorem \ref{ore}, that $\nu_p(ind(f))=ind_\ph(f)=$deg$(\ph)\times ind(S)=\frac{(n-1)(\nu_p(m)-1)+d-1}{2}$ as desired.
 \item 
The second point is also true. But the proof presented in \cite{G} is not correct because it  was based on Proposition $3.5$, whose proof in turn was based on Proposition $3.4$, which as we will see below is not correct.
  \end{enumerate}
\item
Proposition $3.4$ is not correct unless we one add some  supplementary  conditions. In order  to show that this result is not correct,  we consider the example $f(x)=x^{14}-m$ with $m$ is a square free integer and $\nu_2(1-m)\ge 3$. Then $\ol{f(x)}=(x^7-1)^2$ and $x^7-1=(x+1)= (x^3  + x^2  + 1) (1 + x) (x^3  + x + 1)$ in $\F_2[x]$.
For $\ph=x^3+x+1$, we have
$f(x)=(x^2-4)\ph^4+(16-4x^2+10x)\ph^3-(19+4x^2-30x)\ph^3)\ph^2+(6+16x^2+24x)\ph+(1 -m-8x^2-4x)$.
So, $\nph{f}$ has a single side joining the points $(0,2)$, $(1,1)$, and $(4,0)$. This contradicts the claim of Proposition $3.4$, which says that $\nph{f}=S_1+S_2$ has two sides joining the points $(0,3)$, $(1,1)$, and $(4,0)$. \textcolor[rgb]{0.00,0.00,1.00}{The error comes from the fact that  the author assumed that:  for any monic polynomial $\ph\in \Z[x]$, whose reduction is an irreducible factor of $x^t-m$, there exists a polynomial
$h(x)\in \Z[x]$ such that $\ph(x)h(x) = x^t-m$, where $n=p^rt$ and $p$ does not divide $t$, which is not correct because even when $\ol{\ph}$ divides $\ol{x^t-m}$, it is not necessarily that ${\ph}$ divides ${x^t-m}$ in $\Z[x]$}.
  \item
 The comment given by the author after  \cite[Theorem 1.2]{G}, which says
"This theorem gives a second proof of Theorem 1.1. Namely, we see that $E_p = 1$ if and only if $m$ is
square-free and $m^p\not\equiv m \md{p^2}$ for every primeinteger $p$ dividing $n$" is not correct and can be corrected as follows: "$E_p = 0$ if and only if $m$ is
square-free and $m^p\not\equiv m \md{p^2}$ for every primeinteger $p$ dividing $n$". 
\item
More seriously, I can not understand why the author state the comments after Theorem 1.2, on a tower of number fields, despite the fact that the field $K$ is fixed. Moreover this comment is not used anywhere.
\item
The paper finishes with  an extra example, Example 3.6, in which the author considered the number field defined by  an irreducible polynomial $f(x) = x^{6^3}-m$ and gcd$(m, 6) = 1$. The author claimed to caluclate  the $2$-adic and $3$-adic valuations
of the index, given in  two tables. The results are very strange, in fact the $p$-index of a polynomial is always a non-negative integer  and can never  be a decimal number.
 \end{enumerate}
 We will give in section \ref{fix}, two versions improving the statements of both \cite[Theorem 1.1, 1.2]{G}, show the existence  of an error in \cite[Proposition 3.5]{G} and  in the proof of \cite[Theorem 1.2]{G}. We finishe by Theorem \ref{mono}, in which we give sufficient conditions on $n$ and $m$ in order to have $K$  not monogenic.
\section{{Preliminaries}}
In order make clear the technical tools we use in this paper,  we recall some fundamental facts  on  Newton polygon. 
Let $K=\Q(\al)$ be a  number field generated by $\al$  a complex root   of a monic irreducible polynomial $f(x)\in \Z[x]$ of degree $n$ and $\Z_K$ its ring of integers of $K$.
It is well know that   $\Z_K$ is a free $\Z$-module of rank $n$. Thus the abelian group $\Z_K/\Z[\al]$ is finite. Its cardinal order is called the index of $\Z[\al]$, and denoted $ind(f)=(\Z_K: \Z[\al])$.
 For a rational prime integer $p$, if $p$ does not divide $(\mathbb{Z}_K : \mathbb{Z}[\alpha])$, then a well known theorem of Dedekind says that the  factorization of  $ p \mathbb{Z}_K$ can be derived directly  from the factorization of $\overline{f(x)}$ in $\F_p[x]$. In order to apply this  theorem in an effective way, one needs a
criterion to test  whether  $p$ divides  or not the index $(\Z_K:\Z[\al])$. In $1878$, Dedekind   {proved } the  well known Dedekind's criterion (see  \cite[Theorem 6.1.4]{Co}). 
 When Dedekind's criterion fails, { then a method of Ore  $1928$,  based on Newton polygons, can be used} in order to evaluate the  index $(\Z_K:\Z[\al])$, the
absolute discriminant, and the prime ideal factorization of the rational primes into powers of prime ideals of 
a number field $K$ (see \cite{MN, O}).
{ In case Ore's method fails, then an algorithm of  Guardia, Montes, and Nart \cite{GMN},
 using  high order Newton polygons can be used. Such an algorithm } gives after a finite number of iterations a complete answer on the index $(\Z_K:\Z[\al])$, the absolute discriminant $d_K$, and the factorization of $p\Z_K$.
  Let $p$ be a prime integer and $\ol{f(x)}=\prod_{i=1}^r \ol{\ph_i(x)}^{l_i}$ modulo $p$ the factorization of $\ol{f(x)}$ into powers of monic irreducible coprime polynomials of $\F_p[x]$. Recall  the well known  Dedekind's criterion.
  In $1878$, Dedekind { proved }:
  \begin{thm}\label{Ded}$($Dedekind's criterion \cite[Theorem 6.1.4]{Co} and \cite{R}$)$\\
 For a number field $K$ generated by $\al$ a complex root of a monic irreducible  polynomial $f(x)\in \Z[x]$ and a rational prime integer $p$, let $\overline{f}(x)=\prod_{i=1}^r\overline{\ph_i}^{l_i}(x)\md{p}$  be the factorization of   $\overline{f}(x)$ in $\F_p[x]$, where the polynomials $\ph_i\in\Z[x]$ are monic with their reductions irreducible over $\F_p$ and GCD$(\overline{\ph_i},\overline{\ph_j})=1$ for every $i\neq j$. If we set
$M(x)=\cfrac{f(x)-\prod_{i=1}^r{\ph_i}^{l_i}(x)}{p}$, then $M(x)\in \Z[x]$ and the following statements are equivalent:
\begin{enumerate}
\item[1.]
$p$ does not divide the index $(\Z_K:\Z[\al])$.
\item[2.]
For every $i=1,\dots,r$, either $l_i=1$ or $l_i\ge 2$ and $\overline{\ph_i}(x)$ does not divide $\overline{M}(x)$ in $\F_p[x]$.
\end{enumerate}
\end{thm} 
 When Dedekind's criterion fails, that is,  $p$ divides the index $(\z_K:\z[\alpha])$  for every generator $\al$ of $K$, {  then} it is not possible to obtain the prime ideal factorization of $p\z_K$ by{ Dedekind's theorem}.
In 1923, Ore developed   an alternative approach
for obtaining the index $(\z_K:\z[\alpha])$, the
absolute discriminant, and the prime ideal factorization of the rational primes in
a number field $K$ by using Newton polygons (see \cite{MN, O}). For more details on Newton polygon techniques, we refer to \cite{El, GMN}.
 For any  prime integer  $p$,{ let $\nu_p$ be the $p$-adic valuation of $\Q$, $\Q_p$ its $p$-adic completion, and $\Z_p$ the ring of $p$-adic integers. Let  $\nu_p$  { be } the Gauss's extension of $\nu_p$ to $\Q_p(x)$; $\nu_p(P)=\mbox{min}(\nu_p(a_i), \, i=0,\dots,n)$ for any polynomial $P=\displaystyle\sum_{i=0}^na_ix^i\in\Q_p[x]$ and extended by $\nu_p(P/Q)=\nu_p(P)-\nu_p(Q)$ for every nonzero polynomials $P$ and $Q$ of $\Q_p[x]$.} Let 
$\phi\in\z_p[x]$ be a   monic polynomial  {\it whose reduction} is irreducible  in
$\F_p[x]$, let $\fph$ be 
the field $\frac{\F_p[x]}{(\overline{\phi})}$. For any
monic polynomial  $f(x)\in \z_p[x]$, upon  the Euclidean division
 by successive powers of $\ph$, we  expand $f(x)$ as follows:
$f(x)=\displaystyle\sum_{i=0}^la_i(x)\phi(x)^{i},$ called    the $\phi$-expansion of $f(x)$
 (for every $i$, deg$(a_i(x))<$
deg$(\phi)$). 
The $\ph$-Newton polygon of $f(x)$ with respect to $p$, is the lower boundary convex envelope of the set of points $\{(i,\nu_p(a_i(x))),\, a_i(x)\neq 0\}$ in the Euclidean plane, which we denote by $\nph{f}$.  The $\ph$-Newton polygon of $f$, is the process of joining the obtained edges  $S_1,\dots,S_r$ ordered by   increasing slopes, which  can be expressed as $\nph{f}=S_1+\dots + S_r$. 
For every side $S_i$ of $\nph{f}$, the length of $S_i$, denoted   $l(S_i)$  is the  length of its projection to the $x$-axis and its height , denoted   $h(S_i)$  is the  length of its projection to the $y$-axis. {Let $d(S_i)=$GCD$(l(S_i), h(S_i))$ be the ramification degree of $S$.
 The principal $\ph$-Newton polygon of ${f}$},
 denoted $\npp{f}$, is the part of the  polygon $\nph{f}$, which is  determined by joining all sides of negative  slopes.
  For every side $S$ {of} $\npp{f}$, with initial point $(s, u_s)$ and length $l$, and for every 
$0\le i\le l$, we attach   the following
{\ residual coefficient} $c_i\in\fph$ as follows:
$$c_{i}=
\displaystyle\left
\{\begin{array}{ll} 0,& \mbox{ if } (s+i,{\it u_{s+i}}) \mbox{ lies strictly
above } S\\
\left(\dfrac{a_{s+i}(x)}{p^{{\it u_{s+i}}}}\right)
\,\,
\md{(p,\phi(x))},&\mbox{ if }(s+i,{\it u_{s+i}}) \mbox{ lies on }S,
\end{array}
\right.$$
where $(p,\phi(x))$ is the maximal ideal of $\z_p[x]$ generated by $p$ and $\ph$. 
Let $\la=-h/e$ be the slope of $S$, where  $h$ and $e$ are two positive coprime integers. Then  $d=l/e$ is the degree of $S$.  Notice that, 
the points  with integer coordinates lying{ on} $S$ are exactly $$\displaystyle{(s,u_s),(s+e,u_{s}-h),\cdots, (s+de,u_{s}-dh)}$$ Thus, if $i$ is not a multiple of $e$, then 
$(s+i, u_{s+i})$ does not lie in $S$, and so $c_i=0$. Let
{$$f_S(y)=t_dy^d+t_{d-1}y^{d-1}+\cdots+t_{1}y+t_{0}\in\fph[y],$$} called  
the residual polynomial of $f(x)$ associated to the side $S$, where for every $i=0,\dots,d$,  $t_i=c_{ie}$.\\

    Let $\npp{f}=S_1+\dots + S_r$ be the principal $\ph$-Newton polygon of $f$ with respect to $p$.\\
     We say that $f$ is a $\ph$-regular polynomial with respect to $p$, if  $f_{S_i}(y)$ is square free in $\fph[y]$ for every  $i=1,\dots,r$. \\
      The polynomial $f$ is said to be  $p$-regular  if $\overline{f(x)}=\prod_{i=1}^r\overline{\ph_i}^{l_i}$ for some monic polynomials $\ph_1,\dots,\ph_t$ of $\Z[x]$ such that $\ol{\ph_1},\dots,\ol{\ph_t}$ are irreducible coprime polynomials over $\F_p$ and    $f$ is  a $\ph_i$-regular polynomial with respect to $p$ for every $i=1,\dots,t$.
 \smallskip
 
The  theorem of Ore plays  a fundamental key for proving our main Theorems:\\
  Let $\ph\in\Z_p[x]$ be a monic polynomial, with $\overline{\ph(x)}$ is irreducible in $\F_p[x]$. As defined in \cite[Def. 1.3]{EMN},   the $\ph$-index of $f(x)$, denoted by $ind_{\ph}(f)$, is  deg$(\ph)$ times the number of points with natural integer coordinates that lie below or on the polygon $\npp{f}$, strictly above the horizontal axis,{ and strictly beyond the vertical axis} (see $Figure\ 1$).
  
  \begin{figure}[htbp] 
\centering
\begin{tikzpicture}[x=1cm,y=0.5cm]
\draw[latex-latex] (0,6) -- (0,0) -- (10,0) ;
\draw[thick] (0,0) -- (-0.5,0);
\draw[thick] (0,0) -- (0,-0.5);
\node at (0,0) [below left,blue]{\footnotesize $0$};
\draw[thick] plot coordinates{(0,5) (1,3) (5,1) (9,0)};
\draw[thick, only marks, mark=x] plot coordinates{(1,1) (1,2) (1,3) (2,1)(2,2)     (3,1)  (3,2)  (4,1)(5,1)  };
\node at (0.5,4.2) [above  ,blue]{\footnotesize $S_{1}$};
\node at (3,2.2) [above   ,blue]{\footnotesize $S_{2}$};
\node at (7,0.5) [above   ,blue]{\footnotesize $S_{3}$};
\end{tikzpicture}
\caption{    \large  $\npp{f}$.}
\end{figure}
In the example of $Figure\ 1$, $ind_\ph(f)=9\times$deg$(\ph)$ and if $\nph{f}^+=S$ has a single side, then $ind_\ph(f)=\mbox{deg}(\ph)\times ind(S)=\mbox{deg}(\ph)\times\frac{(l-1)(h-1)+d-1}{2}$, where $l$ is the length of $S$, $h$ is its height, and $d=$gcd$(l,h)$ (see \cite{GMN}).\\
\smallskip

  Now assume that $\overline{f(x)}=\prod_{i=1}^r\overline{\ph_i}^{l_i}$ is the factorization of $\overline{f(x)}$ in $\F_p[x]$, where every $\ph_i\in\Z[x]$ is monic polynomial, with $\overline{\ph_i(x)}$ is irreducible in $\F_p[x]$, $\overline{\ph_i(x)}$ and $\overline{\ph_j(x)}$ are coprime when $i\neq j$ and $i, j=1,\dots,t$.
For every $i=1,\dots,t$, let  $N_{\ph_i}^+(f)=S_{i1}+\dots+S_{ir_i}$ be the principal  $\ph_i$-Newton polygon of $f$ with respect to $p$. For every $j=1,\dots, r_i$,  let $f_{S_{ij}}(y)=\prod_{k=1}^{s_{ij}}\psi_{ijk}^{a_{ijk}}(y)$ be the factorization of $f_{S_{ij}}(y)$ in $\F_{\ph_i}[y]$. 
  Then we have the following index theorem of Ore (see \cite[Theorem 1.7 and Theorem 1.9]{EMN}, \cite[Theorem 3.9]{El}, and{\cite[pp: 323--325]{MN}}).
 \begin{thm}\label{ore} $($Theorem of Ore$)$
  \begin{enumerate}
 \item
  $$\nu_p(ind(f))=\nu_p((\z_K:\z[\al]))\ge \sum_{i=1}^t ind_{\ph_i}(f).$$  The equality holds if $f(x)$ is $p$-regular.  
\item
If  $f(x)$ is $p$-regular, then
$$p\Z_K=\prod_{i=1}^t\prod_{j=1}^{r_i}
\prod_{k=1}^{s_{ij}}\p^{e_{ij}}_{ijk},$$ 
where{$e_{ij}=l_{ij}/d_{ij}$, $l_{ij}$ is the length of $S_{ij}$,  $d_{ij}$ is the ramification degree}
 of   $S_{ij}$, and $f_{ijk}=\mbox{deg}(\ph_i)\times \mbox{deg}(\psi_{ijk})$ is the residue degree of $\p_{ijk}$ over $p$.
 \end{enumerate}
\end{thm}
 
  In \cite{GMN}, Guàrdia, Montes, and  Nart introduced  the notion of $\phi$-admissible expansion used in order to treat some special cases when the $\phi$-expansion is not obvious. Let
\begin{equation}
\label{eq1}
f(x)=\displaystyle\sum_{i=0}^nA_i'(x)\phi(x)^i,\quad A_i'(x)\in \mathbb{Z}[x],
\end{equation}
be a $\phi$-expansion of $f(x)$, which does not necessarily  satisffy  deg$(A_i')<$deg$(\phi)$. By analogous to the definition of Newton polygon of $f$, for every $i=0,\dots, n$, let  $u_i'=\nu_p(A_i'(x))$, and  $N'$ be the lower boundary of the convex envelope of the set of points $\{(i,u_i')\,\mid\,0\leq i\leq n,\,u_i'\neq\infty\}$. To any $i=0,\dots,n$, we attach the residue coefficient as follows:
$$
c_i'=\left\{
\begin{array}{ll}
0,&\text{if }(i,u_i')\text{ lies above }N',\\
\left(\frac{A_i'(x)}{p^{u_i'}}\right)\md{p,\phi(x)},& \text{if }(i,u_i')\text{ lies on }N'.
\end{array}
\right.
$$
Likewise, for any side $S'$ of $N'$, we can define the residual polynomial associated to $S$ and denoted $f_{S'}(y)$ (similarely to the residual polynomial $f_{S}(y)$). We say that the $\phi$-expansion $(\ref{eq1})$ is admissible if $c_i'\neq 0$ for each abscissa $i$ of a vetex of $N'$ and it is obvious to see that if $\ol{\ph(x)}$ does not divide $\ol{(\frac{A_i'(x)}{p^{u_i'}})}$, then the $\phi$-expansion $(\ref{eq1})$ is admissible.  For more details, we refer to \cite{GMN}.
\begin{lem}
	If a $\phi$-expansion of $f(x)$ is admissible, then $N'=N_{\phi}^-(f)$ and $c_i'=c_i$. In particular, for any side $S$ of $N'$ we have $R_{\lambda}'(f)(y)=R_{\lambda}(f)(y)$ up to multiply by a nonzero coefficient of $\fph$.
\end{lem}
Again, if  $f(x)$ is not $p$-regular, that is theorem of Ore fails, then in order to complete the calcul of index, the factorization of $f(x)$, and the absolute discriminant of $K$, Guardia, Montes, and Nart introduced the notion of high order Newton polygon. They showed,  thanks to a theorem of index, that  after a finite number of iterations this process yields all monic irreducible factors of $f(x)$, all prime ideals of $\Z_K$ lying above a prime integer $p$, the index $(\Z_K:\Z[\al])$, and the absolute discriminant of $K$. { We recall here some fundamental techniques of Newton polygon of high order. For more details, we refer to \cite{GMN}. As introduced in \cite{GMN},  a type of order $r - 1$ } is  a data ${\bf{t}}  = (g_1(x), -\la_1, g_2(x), -\la_2,\dots, g_{r-1}(x), -\la_{r-1},\psi_{r-1}(x))$,
		where every $g_{i}(x)$ is a monic polynomial in $\z_p[x]$, $\la_i\in \q^+$, and
		$\psi_{r-1}(y)$ is a polynomial over a  finite field of {$\displaystyle p^{H}$ and $H=\displaystyle\prod_{i=0}^{r-2}f_i$} elements,  with  $f_i=\mbox{deg}(\psi_{i}(x))$, satisfying the following recursive properties:
		\begin{enumerate}
			\item 
			$g_{1}(x)$ is irreducible modulo $p$,  $\psi_0(y) \in \F[y]$ ($\F_0=\F_p$) be the polynomial obtained
			by reduction of $g_{1}(x)$ modulo $p$, and  $\F_1 := \F_0[y]/(\psi_0(y)))$.
			\item
			For every $i=1,\dots,r-1$, the Newton polygon of $i^{th}$ order, $N_i(g_{i+1}(x))$,  has a single  sided of slope $-\la_i$.
			\item 
			For every $i=1,\dots,r-1$, the residual polynomial of $i^{th}$ order, $R_i(g_{i+1})(y)$ is an
			irreducible polynomial in $\F_i[y]$, $\psi_i(y)\in\F_i[y]$ be the monic polynomial
			determined by $R_i(\ph_{i+1})(y)\simeq \psi_i(y)$ (are equal up to multiplication by a nonzero element of $\F_i$, and  $\F_{i+1}= \F_i[y]/(\psi_i(y))$. Thus, $\F_0\subset \F_1\subset \dots\subset \F_r$ is a tower of finite fields.
			\item
			For every $i=1,\dots,r-1$, $g_{i+1}(x)$ has minimal degree among all monic polynomials
			in $\z_p[x]$ satisfying (2) and (3).
			\item 
			$\psi_{r-1}(y) \in\F_{r-1}[y]$ is a monic irreducible polynomial,  {$\psi_{r-1}(y)\neq y$}, and $\F_{r}= \F_{r-1}[Y]/(\psi_{r-1}(y))$.
		\end{enumerate}
		Here  the field
		$\F_i$ should not be confused with the finite field of $i$ elements.\\
		As for every $i=1,\dots,r-1$, the residual polynomial of the $i^{th}$ order, $R_i(g_{i+1})(y)$ is an
		irreducible polynomial in $\F_i[y]$, by  theorem of the product in order $i$, the polynomial $g_{i}(x)$ is irreducible in $\z_p[x]$. Let  $\om_0=[\nu_p,x,0]$  {be} the Gauss's extension of $\nu_p$ to $\q_p(x)$. As  for every $i=1,\dots,r-1$, the residual polynomial of the $i^{th}$ order, $R_i(g_{i+1})(y)$ is an
		irreducible polynomial in $\F_i[y]$, then according to MacLane notations  {and definitions} (\cite{Mc}),  $g_{i+1}(x)$ induces  a valuation  on $\q_p(x)$,  denoted by  $\om_{i+1}=[e_i\om_{i},g_{i+1},\la_{i+1}]$, where $\la_i=h_i/e_i$,   $e_i$ and  $h_i$ are  positive coprime integers. The valuation   $\om_{i+1}$ is called the augmented valuation of $\nu_p$  with respect to $\ph$ and $\la$ is  defined over   $\q_p[x]$  as follows:   
		$$
		\om_{i+1}(f(x))=\mbox{min}\{e_{i+1}\om_i(a_j^{i+1}(x))+jh_{i+1},\,j=0,\dots,n_{i+1}\},$$ 
		where $f(X)=\displaystyle\sum_{j=0}^{n_{i+1}}a_j^{i+1}(x)g_{i+1}^j(x)$ is the $g_{i+1}(x)$-expansion of $f(x)$. According to the terminology  in \cite{GMN}, the valuation  $\om_r$ is called the  $r^{th}$-order valuation associated to the data ${\bf{t}}$.
		For every order $r\ge 1$, the $g_r$-Newton polygon of $f(x)$, with respect  to  the valuation $\om_r$ is the lower boundary of the convex envelope of  the set of  points $\{(i,\mu_i), i=0,\dots, n_r\}$ in the  {Euclidean} plane, where $\mu_i=\om_r(a^r_i(X)g_r^i(x)))$.  
		{The following are the relevant theorems from Montes-Guardia-Nart's work (high order Newton polygon):
			\begin{thm}$($\cite[Theorem 3.1]{GMN}$)$\label{HNP}\\
				Let $f \in \z_{p}[x]$ be a monic polynomial such that $\overline{f(x)}$ is a positive power of
				$\overline{\ph}$. If  $N_r(f)=S_1+\dots+S_g$ has $g$ sides, then we can
				split $f(x)=f_1\times\dots\times f_g(x)$ in $\z_p[X]$, such that $N_r(f_i)=S_i$ and
				$R_r({f_i})(y)=R_r({f})(y)$ up to  multiplication by a nonzero
				element of $\F_r$ for every $i=1,\cdots,g$.
			\end{thm}
			\begin{thm}$($\cite[Theorem 3.7]{GMN}$)$\label{HNR}\\
				Let $f \in \z_{p}[X]$ be a monic polynomial such that  $N_r(f)=S$ has a single side of finite slope $-\la_r$.
				If $R_r({f})(y)=\displaystyle\prod_{i=1}^t\psi_i(y)^{a_i}$ is the factorization in $\F_r[y]$, then  $f(x)$ splits as $f(x)=f_1(x)\times\cdots\times f_t(x)$ in $\z_p[x]$ such that $N_r({f_i})=S$ has  a single side of slope $-\la_r$ and $R_r({f_i})(y)=\psi_i(y)^{a_i}$  up to  multiplication by a nonzero
				element of $\F_r$ for every $i=1,\cdots,t$.
			\end{thm}
			In \cite[Definition 4. 15]{GMN}, the authors introduced the notion of $r^{th}$-order index of a monic polynomial $f\in \Z[x]$ as follows: For a fixed data $${\bf{t}}  = (g_1(x), -\la_1, g_2(x), -\la_2,\dots, g_{r-1}(x), -\la_{r-1},\psi_{r-1}(x)),$$ let $N_r(f)$ be the Newton polygon of $r^{th}$-order with respect to the data  {\bf{t}} and   $ind_t(f)=f_0\cdots f_{r-1}ind(N_r(f))$, where  and $ind(N_r(f))$ is the index of the polygon $N_r(f)$; the number of points with natural integer coordinates that lie below or on the polygon $\npp{f}$, strictly above the horizontal line of equation $y=\om_r(f)$,{ and strictly beyond the vertical axis}. The $r^{th}$-order index of $f$ is defined by 			
			$$ind_r(f)=\displaystyle\sum_{t\in {\bf{t}}}ind_t(f).$$
			In particular, if $\ol{f(x)}$ is a power of $\ol{\ph(x)}$, then $ind_1(f)=ind_\ph(f)$.
			In \cite[Theorem 4. 18]{GMN}, they showed the following index formula which generalizes the theorem of index of Ore:
$$ ind(f)\ge  ind_1(f) +\dots + ind_r(f),$$ and 
the equality holds if and only if $ind_{r+1}(f) = 0$. Recall that by definition $ind(N_{r+1}(f))=0$ if and only if $N_{r+1}(f)$ has a single side of length $1$ or height $1$. By  \cite[Lemma 2. 17]{GMN} $(2)$, if $R_r(f)$ is square free, then the length of $N_r(f)$
 is $1$. Thus if $R_r(f)$ is square free, then $ind_{r+1}(f)=0$, and so the equality  $ ind(f)\ge  ind_1(f) +\dots + ind_r(f)$ holds. 	
\section{Some improvements and new results}\label{fix}
Throught this section unless otherwise noted $f(x)=x^n-m\in \Z[x]$ is an irreducible polynomial such that  $n\ge 2$ and $\nu_p(m)<n$ for every prime integer $p$. Let $K=\Q(\al)$ be the number field generated by a complex root  $\al$ of $f(x)$.\\
\smallskip

In order to fix the error in the remark of \cite[Theorem 1.1]{G}, we show  the following lemma:
\begin{lem}\label{2^k}
Let $k$ be a non-negative integer and $f(x)=x^{2^k}+1\in \Z[x]$. Then $f(x)$ is irreducible over $\Q$. 
\end{lem}
\begin{proof}
Let $\ph=x-1$. Then $\ol{f(x)}=\ol{\ph(x)}^{2^k}$ in $\F_2[x]$ and  $f(x)=\displaystyle\sum_{j=1}^{2^k}\binom{2^k}{j}\phi^j(x)+2$  is the $\ph$-expansion of $f(x)$. Then  $\nph{f}=S$ has a single side of height $1$, with respect to the valuation $\nu_2$. Let $g(x)=f(x+1)$. Then  $g(x)$ is $2$-Eisenstein.  Thus $g(x)$ is irreducible over $\Q$, and so $f(x)$ is irreducible.
\end{proof}
Remark that as Dedekind's criterion characterizes the integral closedness of $\Z[\al]$,  the converse of \cite[Theorem 1.1]{G} holds. So,  the following is an improvement on it:
\begin{thm}\label{intclos}
$\Z[\al]$ is the ring of integers of $K$ if and only if  $\nu_p(m^p- m)=1$ for every prime integer $p$ dividing $n\cdot m$.
\end{thm}
Notice also,  as it is showen by Corollary \ref{ns}, the statement in Theorem \ref{intclos}  does not characterize the monogeneity of $K$.
\begin{cor}\label{ns}
Let $f(x)=x^n-a^v\in \Z[x]$ with $a\neq \mp 1$  a square free integer and $v$  a positive integer which is coprime to $n$.
If for every prime integer $p$ dividing $n$, $\nu_p(a^{p-1}-1)=1$, then $K$ is monogenic.
\end{cor}
\begin{proof}
Since gcd$(v,n)=1$, let $i,j$ be two non-negative integers satisfying $vi-nj=1$ with $i<n$. Let also $\th=\frac{\al^i}{a^j}$. Then $\th\in K$ and $\th^n=\frac{a^{vi}}{a^{nj}}=a$. Since $a\neq \mp 1$ is a square free integer, $g(x)=x^n-a$ is an Eisenstein polynomail, and so is the  minimal polynomial of $\th$ over $\Q$.  Thus $\th$ is a primitive element of  $K$, which is  a root of $g(x)=x^n-a$ with $a\neq \pm 1$ is a square free integer. By applying Theorem \ref{intclos}, $\Z_K=\Z[\th]$.
\end{proof} 
\smallskip

Notice that \cite[Proposition 3.4, 3.5]{G} can be ajusted as follows: 
\begin{prop}\label{NP}
Let $f(x)=x^n-m\in \Z[x]$ be an irreducible polynomial and $p$ a prime integer which divides $n$ and does not divide $m$. Let $n=p^rt$ in $\Z$ with $p$ does not divide $t$. Then $\ol{f(x)}=\ol{(x^t-m)}^{p^r}$. Let $v=\nu_p(m^p-m)$ and $\ph\in \Z[x]$ be a monic polynomial, whose reduction modulo $p$ divides $\ol{f(x)}$. 
 Let us denote $(x^t-m)=\ph(x)Q(x)+R(x)$. Then $\nu_p(R)\ge1$. 
 \begin{enumerate}
 \item
 If $\nu_p(m^p-m)\le r$ or $\ph(x)$ divides $x^t-m$ in $\Z[x]$, then $\npp{f}$ is the lower boundary of the  convex envelope of the set of the points $\{(0,v)\}\cup \{(p^j,r-j), \, j=0,\dots,r\}$.
 \item
 If $\nu_p(m^p-m)\ge r+1$, then $\npp{f}$ is the lower boundary of the  convex envelope of the set of the poiints $\{(0,V)\}\cup \{(p^j,r-j), \, j=0,\dots,r\}$ and $\nu_p(ind(f))=\mbox{deg}(\ph)\times \displaystyle\sum_{j=1}^{\mbox{min}(r,\nu(m^p-m)-1)}p^j$ for some integer $V\ge r+1$.
 \item
  $\nu_p(ind(f))=\mbox{deg}(\ph)\times \displaystyle\sum_{j=1}^{\mbox{min}(r,\nu(m^p-m)-1)}p^j$.
 \end{enumerate}
\end{prop}
  \begin{proof}
Let $f(x)=(x^t-m+m)^{p^r}-m=(\ph(x)Q(x)+R(x)+m)^{p^r}-m=\displaystyle\sum_{j=1}^{p^r}\binom{p^r}{j}(R(x)+m)^{{p^r}-j}Q^j\phi^j(x)+(R(x)+m)^{p^r}-m$. Since $x^t-m$ is separable over $\F_p$, $\ol{\ph}$ does not divide $\ol{Q(x)}$. Thus if $R=0$, this $\ph$-expansion is admissible and so $\npp{f}$ is the lower boundary of the  convex envelope of the set of the points $\{(0,v)\}\cup \{(r,r-j), \, j=0,\dots,r-1\}$. If $R\neq 0$, then let
$(R(x)+m)^{p^r}-m^{p^r}=\displaystyle\sum_{k\ge 0}r_j\phi^j(x)$ be the $\ph$-expansion of $(R(x)+m)^{p^r}-m^{p^r}$.  Then
$f(x)=\dots+\displaystyle\sum_{j=1}^{p^r}\binom{p^r}{j}((R(x)+m)^{{p^r}-j}Q^j+r_j)\phi^j(x)+(r_0+m^{p^r}-m)$ is the $\ph$-expansion of $f(x)$. 
Since $\nu_p(R)\ge 1$, $(R(x)+m)^{p^r}-m^{p^r}=\displaystyle\sum_{j=1}^{p^r}\binom{p^r}{j}m^{{p^r}-j}R(x)^j+m^p$, and 
 $\nu_p(\binom{p^r}{j}m^{{p^r}-j}R(x)^j)\ge r-\nu_p(j)+j\ge  r+1$, we get  $\nu_p((R(x)+m)^{p^r}-m^{p^r})\ge r+1$, and so $\nu_p(r_k)\ge r+1$ for every $k\ge 0$.  As for every $j=0,\dots, p^r$, $\nu_p(\binom{p^r}{j})\le r$,  $\nu_p(\binom{p^r}{j}((R(x)+m)^{{p^r}-j}Q^j+r_j))= \nu_p(\binom{p^r}{j})$ for every $j=1,\dots, p^r$. It follows that if $\nu_p(m^p-m)\le r<r+1$, then $\npp{f}$ is the lower boundary of the  convex envelope of the set of the points $\{(0,\nu_p(m^p-m))\}\cup \{(r,r-j), \, j=0,\dots,r-1\}$. But if  $\nu_p(m^p-m)\ge r+1$, then $\npp{f}$ is the lower boundary of the  convex envelope of the set of the points $\{(0,V)\}\cup \{(r,r-j), \, j=0,\dots,r-1\}$, where $V=\nu_p(m^p-m+r_{0})\ge \mbox{min}(\nu_p(m^p-m),\nu_p(r_0))\ge r+1$.\\  
 Let $y=1,\dots, \mbox{min}(\nu_p(m^p-m)-1,r)$. As the number of points with positive coordinates $(i,y)$ lying below the polygon $\npp{f}$ is  $p^{r-y}$, then  $ind(\npp{f})= \displaystyle\sum_{y=1}^{\mbox{min}(r,\nu(m^p-m)-1)}p^y$ and  $ind_\ph(f)=\mbox{deg}(\ph)\times \displaystyle\sum_{j=1}^{\mbox{min}(r,\nu(m^p-m)-1)}p^{r-j}$. Since $\npp{f}$ is the join of sides  of degree $1$, except  for $p=2$, my be the first one is of degree $2$, with associated residual polynomial of $f(x)$ is $y^2+y+1$, which  is irreducible over $\F_2=\fph$, the polynomial $f(x)$ is $p$-regular and thus $\nu_p(ind(f))=ind_\ph(f)=\mbox{deg}(\ph)\times \displaystyle\sum_{j=1}^{\mbox{min}(r,\nu(m^p-m)-1)}p^{r-j}$ as desired.
\end{proof}
 An interesting question is "under which weaker conditions on $n,m$, and $p$ such that $p$ divides $m$, we can keep the equality $ \nu_p(indf)) = \frac{(n-1)(v-1) + d-1}{2} $?"
The answer is given by the following proposition:
\begin{prop}
If gcd$(p,\nu_p(m), n)=p$ for some prime integer $p$, then $\nu_p(ind(f))>\frac{(n-1)(\nu_p(m)-1)+d-1}{2}$.
\end{prop}
\begin{proof}
First by Proposition \ref{NP} and  Theorem \ref{ore}, we have $\nu_p(ind(f))\ge \frac{(n-1)(\nu_p(m)-1)+d-1}{2}$. Moreover if
gcd$(p,\nu_p(m), n)=1$, then the equality holds. Now assume that gcd$(p,\nu_p(m), n)=p$.
 Let $n=p^rt$ and  $\nu_p(m)=p^su$ with  $p$ does not $ut$. 
For $\ph=x$, $\ol{f(x)}=\ol{\ph(x)}^n$  in $\F_p[x]$, $\nph{f}=S$ has a single side of slope $-\la_1=-\frac{u}{p^{r-s}t}$ and $f_S(y)=(y-m_p)^{p^v}$, where $v=\mbox{min}(s,r)$, we conclude that  $f_S(y)$  is not square free, and so we have to use second order Newton polygon. Let $\la_1=h_1/e_1$ with $e_1$ and $h_1$ are two coprime positive integers. According to Nart's notations in \cite{GMN}, let  $\ph_2(x)=x^{e_1}-p^um_p$ and $V_2$ be the valuation of second order Newton polygon  associated to the data $(x,\la, y-m_p,\ph_2)$; $V_2(\sum_{i=0}^ka_ix^i)=\mbox{min}(e_1\nu_p(a)+i\la,\, i=0,\dots,k)$. Let $f(x)=(x^{p^{r-s}t}-p^um_p+p^um_p)^{p^s}-m=\displaystyle\sum_{i=1}^{p^s}\binom{p^s}{i}(p^um_p)^{p^s-i}\ph_2^i(x)+p^{{p^s}u}(m_p^{p^s}-m_p)$ be the $\ph_2$-expansion of $f(x)$. Since $\nu_p((m_p^{p^s}-m_p))\ge 1$, then $V_2((m_p^{p^s}-m_p))\ge e_1$. Thus if $e_1>1$, then $ind_2(f)\ge 1$, and so $\nu_p(ind(f))\ge ind_1(f)+ind_2(f)>ind_1(f) =\frac{(n-1)(v-1)+d-1}{2}$.
It follows that: 
\begin{enumerate}
\item
If $s<r$, then $e_1=p^{r-s}t>p>$, and so $\nu_p(ind(f))>\frac{(n-1)(v-1)+d-1}{2}$.
 \item
 If $s\ge r$ and $t>1$, then $e_1=t>1$ and $\nu_p(ind(f))>\frac{(n-1)(v-1)+d-1}{2}$.
 \item
  The case  $s\ge r$ and $t=1$ is excluded because $\nu_p(m)$ is assumed to be less than $n$.
\end{enumerate}
\end{proof}
The following theorem gives a condition on $f(x)$ in order to have $K$ is not monogenic.
\begin{thm}\label{mono}
Let $n=p^rt$, $m=p^su$, and $m_p=\frac{m}{p^{s}}=u$ with $p$ does not divide $tu$.
If one of the following holds:
\begin{enumerate}
\item
$p$ is odd and $\nu_p(1- m)\ge p+1$.
\item
$pt$ is odd  and $\nu_p(1+ m)\ge p+1$.
\item
$p=3$, $t$ is even and $\nu_3(1+ m)\ge 4$.
\item
$p=2$, $2$ does not divide $m$, $r=2$, and $\nu_2(m^{p-1}-1)\ge 4$.
\item
$p=2$, $2$ does not divide $m$, $r\ge 3$, and $\nu_2(m^{p-1}-1)\ge 5$.
\item
$p$ is odd, $r=s$,  $t>1$, gcd$(t,p-1)=1$, $r\ge p$, and $\nu_p(m_p^{p}-m_p)\ge p+1$.
\item
$p=2$, $r=s$, $r=2$, and $\nu_2(m_2-1)\ge 4$. 
\item
$p=2$, $r=s$,  $r\ge 3$, and $\nu_2(m_2-1)\ge 5$. 
 \end{enumerate}
  then $K$ is not monogenic.
\end{thm}
In order to prove  Theorem \ref{mono},  we need the following two lemmas. The  first one is an immediate consequence of {Dedekind's} theorem. The second one follows from the \cite[Corollary 3.8]{GMN}. 
\begin{lem} \label{comindex}
 Let  $p$ be  rational prime integer and $K$  a number field. For every positive integer $f$, let $P_f$ be the number of distinct prime ideals of $\Z_K$ lying above $p$ with residue degree $f$ and $N_f$  the number of monic irreducible polynomials of  $\F_p[x]$ of degree $f$.
{ If $ P_f > N_f$ for some
positive integer $f$}, then for every generator $\th\in \Z_K$ of $K$, $p$ divide the index  $(\Z_K:\Z[\th])$.
\end{lem}
\begin{lem} \label{residue}
 Let $p$ be a prime integer, $f(x)\in \Z_p[x]$ a monic  polynomial such that $\ol{f(x)}$ is a  power of $\ol{\ph(x)}$ for some monic polynomial $\ph\in \Z_p[x]$, whose reduction is irreducible over $\F_p$, $\nph{f}=S$ has a single side of slope $-\la_1$, $f_S(y)=\psi^a(y)$ for some monic irreducible polynomial $\psi\in \fph[y]$, and $N_2(f)=T$ has a single side of slope $-\la_2$. Let $e_i$ be the smallest positive integer satisfying $e_i\la_i\in \Z$. If $R_2(f)$ is irreducible over $\F_2=\frac{\fph[y]}{(\psi(y))}$, then $f(x)$ is irreducible over $\Q_p$. Let $\p$ be the unique prime ideal of $\Q_p(\beta)$ lying above $p$, where $\beta$ is a root of $f(x)$. Then $e(\p)=e_1e_2$ is  the ramification index of $\p$  and $f(\p)=\mbox{deg}(\ph)\times \mbox{deg}(\psi)\times \mbox{deg}(R_2(f))$ is its residue degree.
\end{lem}
\begin{proof} of Theorem \ref{mono}.\\
\begin{enumerate}
\item
 Since $x^t-1$ is separable over $\F_p$, if $m\equiv 1\md{p}$, then $\ol{f(x)}=(x^t-1)^{p^r}=(\ol{\ph(x)Q(x)})^{p^r}$ in $\F_p[x]$, where $\ph=x-1$ and $Q(x)\in \Z[x]$ with $\ol{\ph}$ does not divide $\ol{Q(x)}$ in $\F_p[x]$. Since $x-1$ divides $x^t-1$ in $\Z[x]$, by the first part of Proposition \ref{NP}, if $\nu_p(m-1)\ge p+1$ and $r\ge p$, then $\npp{f}$ has at least $p+1$ sides of degree $1$ each one. Thus by Theorem \ref{ore}, there are at least $p+1$ prime ideals of $\Z_K$ lying above $p$ of residue degree $1$ each one. By Lemma \ref{comindex} and  the fact the there are only $p$-monic irreducible polynomial of degree $1$ in $\F_p$, $p$ is a common index divisor of $K$, and so $K$ is not monogenic.
\item
For the same argument, if  $pt$ is odd  and $\nu_p(1+ m)\ge p+1$, then for $\ph=x+1$, $\npp{f}$ has at least $p+1$ sides of degree $1$ each one, and so $p$ is a common index divisor of $K$.
\item
For the same argument, if $p=3$, $t$ is even and $\nu_3(1+ m)\ge 4$, then for $\ph=x^2+1$, $\npp{f}$ has at least $4$ sides of degree $1$ each one. Thus by Theorem \ref{ore}, there at least $4$ prime ideals of $\Z_K$ lying above $3$ with residue degree $2$ each one. The fact that there are only three monic irreducible polynomial  polynomial of degree $2$ in $\F_3[x]$, namely, $x^2+1$, $x^2+x-1$, and $x^2-x-1$, we conclude that $3$ is a common index divisor of $K$.
\item
For $p=2$ and $2$ does not divide $m$, we have $\ol{f(x)}=(x^t-1)^{2^r}=(\ol{\ph(x)Q(x)})^{2^r}$ in $\F_2[x]$, where $\ph=x-1$ and $Q(x)\in \Z[x]$ with $\ol{\ph}$ does not divide $\ol{Q(x)}$ in $\F_2[x]$. Since $x-1$ divides $x^t-1$ in $\Z[x]$, by the first point of Proposition \ref{NP}, we have $\npp{f}$ has at least three sides of degree $1$ each one. Thus by Theorem \ref{ore}, there are at least $3$ prime ideals of $\Z_K$ lying above $2$ of residue degree $1$ each one. Therefore, by Lemma \ref{comindex}, $2$ is a common index divisor of $K$, and so $K$ is not monogenic.
\item
For the last three points,   first $\ol{f(x)}=\ph(x)^{n}$ in $\F_p[x]$ with $\ph=x$. Since $\nph{f}=S$ has a single side of slope $-\la_1=-u/t$ and  $f_S(y)=(y-m_p)^{p^r}$ (because $r=s$), we have to use second order Newton polygon techniques. 
Let $\la_1=h_1/e_1$ with $e_1$and $h_1$ are two coprime positive integers. Then   $e_1=t$, $\ph_2(x)=x^{t}-p^um_p$. Let $f(x)=(x^{t}-p^um_p+p^um_p)^{p^r}-m=\displaystyle\sum_{i=1}^{p^r}\binom{p^r}{i}(p^um_p)^{p^r-i}\ph_2^i(x)+p^{{p^r}u}(m_p^{p^r}-m_p)$  be the $\ph_2$-expansion.  Let also $\om_2$ be the valuation of second order Newton polygon associated to $(x,\frac{\nu_p(m)}{n}, \psi(y)$, where $\psi(y)=y-m_p$. It follows that:
 \begin{enumerate}
\item
 If $p$ is odd, $r\ge p$, and $\nu_p(m_p^{p}-m_p)\ge p+1$, then $N_2(f)=S_1+ S_2+\dots+S_g$, the $\ph_2$-Newton polygon of $f(x)$ with respect to $\om_2$ has $g$ sides $S_1,\, S_2,\,\dots, S_g$ with $g\ge p+1$. Since for every $l(S_{g-i})=p^{i}-p^{i-1}=p^{i-1}(p-1)$ is the length of $S_{g-i}$ and $h(S_{g-i})=t$ is its heigh  and  gcd$(t,p(p-1))=1$, then the side $S_{g-i}$ is of degree $1$ for every $i=0,\dots,p$.
  By Theorems \ref{HNP} and \ref{HNR}, $f(x)=g(x)\times f_1(x)\times \dots\times f_{p+1}(x)$ with 
     deg$(R_2(f_i)(y))=1$ for every $i=1,\dots,g$.
   By Hensel's correspondence, let  $\p_i$ be the prime ideal of $\Z_K$ lying above $p$ and associated to the factor  $f_i(x)$ for every $i=1,\dots,p+1$. Then by Lemma \ref{residue}, $f(\p_i)=1$ for every $i=1,\dots,p+1$. Since there is only $p$ monic irreducible polynomial in $\F_p[x]$, by Lemma \ref{comindex}, $p$ is a common index divisor  of $K$, and so $K$ is not monogenic.
 \item
If $p=2$, $r=2$, and $\nu_2(m_p-1)\ge 4$, then by analogous to the previous  point,
$N_2(f)$ has  exactly $4$ sides of degree gcd$(t,2)=1$ each one.
 Thus there are  $4$ prime ideals of $\Z_K$ lying above $2$ with residue degree $1$ each one. Since there is only $2$ monic irreducible polynomial in $\F_2[x]$, we conclude $2$ is a common divisor index of $K$, and so $K$ is not monogenic.
\item
If $p=2$, $r\ge 3$, and $\nu_2(m_p-1)\ge 5$, then $N_2(f)$  has at least $4$ sides of which at least $3$ are of degree $1$ each one. 
 Thus there are at least $3$ prime ideals of $\Z_K$ lying above $2$ with residue degree $1$ each one. Since there is only $2$ monic irreducible polynomial in $\F_2[x]$, we conclude that $2$ is a common divisor index of $K$, and so $K$ is not monogenic. 
\end{enumerate}
\end{enumerate}
\end{proof}
\section{Examples}
\begin{enumerate}
\item
Let $n\ge 2$ be an integer and $m=(n^*)^u$, where $n^*=\displaystyle\prod_{p\in I}p$, $I$ is the set of positive prime integers dividing $n$, and $u$ is coprime to $n$. Then $f(x)=x^n-m$ is irreducible over $\Q$. Let $K$ be the pure number field  defined by $f(x)$. Then $K$ is monogenic.
\item
Let $f(x)=x^{48}-528$. Then $f(x)$ is $3$-Eisenstein, and so  is irreducible over $\Q$. Let $K$ be the pure number field  defined by $f(x)$. Since  $n=2^4\times 3$,  $m=2^4\times 33$, $m_2=33$, and $\nu_2(m_2-1)=5$, we conclude by Theorem \ref{mono} (8), that $K$ is not  monogenic.
\item
Let $f(x)=x^{135}+2214$. Then $f(x)$ is $2$-Eisenstein, and so  is irreducible over $\Q$. Let $K$ be the pure number field  defined by $f(x)$. Since  $n=3^3\times 5$,  $m=-3^3\times(1+3^4)$, $m_3=-(1+3^4)$, $\nu_3(m_3+1)=4$, and gcd$(5,3\cdot 2)=1$, we conclude by Theorem \ref{mono} (3), that $K$ is not  monogenic.
\item
Let $f(x)=x^{135}-2214$. Then $f(x)$ is $2$-Eisenstein, and so  is irreducible over $\Q$. Let $K$ be the pure number field  defined by $f(x)$. Since  $n=3^3\times 5$,  $m=3^3\times(1+3^4)$, $m_3=(1+3^4)$,  $\nu_3(m_3-1)=4$, and gcd$(5,3\cdot 2)=1$, we conclude by Theorem \ref{mono} (1), that $K$ is not  monogenic. 
\end{enumerate}

\end{document}